\documentclass[12pt]{article}
\usepackage{amsmath, amssymb, verbatim, amsthm, graphicx, titlesec}
\usepackage{hyperref}
\usepackage{multicol}

\usepackage[dvipsnames]{xcolor}

\newtheorem{theorem}{Theorem}

\newtheorem{lemma}[theorem]{Lemma}
\newtheorem{corollary}{Corollary}
\numberwithin{equation}{section}

\newcommand{\Z}{\mathbb{Z}}

\newcommand{\A}{\mathcal{A}}
\newcommand{\AP}[1]{\mathcal{A}_{\left\{ 001,210 \right\}}(#1)}
\newcommand{\APK}[2]{\mathcal{A}_{\left\{ 001,210 \right\}}(#1, #2)}

\title{Two Combinatorial Interpretations of Rascal Numbers}
\author{Amelia Gibbs \and Brian K. Miceli}
\date{%
 Trinity University \\
    \today
}

\begin{document}

\maketitle

\hrule

\begin{abstract}
The main goal of this paper is to assign two combinatorial interpretations to the elements of the Rascal Triangle defined by Angorro et al \cite{ALT}. The first interpretation involves counting ascents in binary words, while the second interpretation involves pattern avoidance in the ascent sequences studied by
Duncan and Steingr\'imsson  \cite{DS}, and related directly to a more recent paper by Baxter and Pudwell \cite{BP}.
\end{abstract}

Angorro et al  \cite{ALT} define a Pascal-like triangular sequence, which they call the \emph{Rascal Triangle}. Accordingly, let's call the numbers in the $n$-th row and $k$-th column of this triangle the \emph{n,k-th Rascal number}, $R_{n,k}$, which corresponds to the OEIS sequence \href{https://oeis.org/A077028}{\texttt{A077028}} \cite{OEIS}. Then the authors define\footnote{They don't technically do it \emph{this} way in \cite{ALT}, because they don't give an explicit indexing, but this is the corresponding recursion based on their description. The same is true for Fleron's note \cite{Fle}.} $R_{0,0} = R_{1,0} = R_{1,1} =1$, and for $n \geq 2$,
\begin{equation}
    R_{n,k} = \frac{R_{n-1,k}R_{n-1,k-1} +1}{R_{n-2,k-1}}, 
    \label{orec}
\end{equation}
where $R_{n,k} = 0$ whenever $k>n$, $n < 0$, or $k<0$. The first few rows of the Rascal Triangle are shown below, where the top row is the 0-th row and the leftmost column is the 0-th column, and any (unshown) entries above, to the left, or to the right of the given entries are equal to 0.

\begin{center}
\begin{tabular}{ccccccc}

1&&&&&& \\ 
1&1&&&&& \\ 
1&2&1&&&& \\ 
1&3&3&1&&& \\ 
1&4&5&4&1&& \\ 
1&5&7&7&5&1& \\ 
1&6&9&10&9&6&1  
\end{tabular}
\end{center}

\noindent Angorro et al then show that $R_{n,k}$ is indeed an integer for all $n,k$, which is not immediately obvious from Equation (\ref{orec}), and they do this by showing, algebraically, that $R_{n,k} = k(n-k) + 1$. For instance we can see in our triangle above that $R_{6,3} = 10$, and it is in fact the case that  $10= 3(6-3) + 1$. In a subsequent note by Fleron \cite{Fle}, it is shown, again algebraically, that the Rascal numbers also satisfy the following linear recurrence, subject to identical initial conditions given above:
\begin{equation}
R_{n,k} = R_{n-1,k} + R_{n-1,k-1} - R_{n-2,k-1} + 1.
\label{fleron}
\end{equation}

\noindent Subsequently, a note by Ashfaque \cite{Ash} gives a single identity involving Rascal numbers and binomial coefficients (our Theorem \ref{trianglesum}, which that author proves algebraically, but we prove combinatorially). Later, an extensive paper by Hotchkiss \cite{Hot}  generalizes the idea of the Rascal Triangle, but this work is also algebraic, and giving a combinatorial sense of these generalized objects is not the goal of that work. And yet, as Rascal numbers may be defined in a way that \emph{so} closely resembles the recursion for binomial coefficients, it seems like some straightforward combinatorial interpretation(s) of Rascal numbers must exist.

%%%%%%%%
%%%%%%%%
%%%%%%%%
%%%%%%%%
%%%%%%%%
%%%%%%%%

\section{Ascents \& Binary Words}

Given any nonegative integer $m$, define $[m] = \{0,1,2,\ldots,m\}$. Let $w = w_1w_2 \cdots w_n$ denote a string of integers from $[m]$, called a \emph{word of length n over} $[m]$, where $w$ is a \emph{binary word} if $w_i \in [1]$ for all $1 \leq i \leq n$. Given $w$, we define an \emph{ascent position} of $w$ to be any index $1 \leq i \leq n-1$ such that $w_i < w_{i+1}$, and we define the \emph{ascent number} of $w$ to be the number of ascent positions of $w$, denoted by asc$(w)$. For example, $w = 2051159858$ is a word of length 10 over $[9]$, and asc$(2\underline{0}51\underline{1}\underline{5}98\underline{5}8) = 4$, where the ascent positions of $w$ have been underlined. Note that the word of length $n = 0$, denoted by $\epsilon$, is called the \emph{empty word}, and asc$(\epsilon) = 0$. From here onward we adopt the notation $x^j$ to denote the letter $x$ repeated $j$ times; for example, we may write the word $u = 11110220005$ as $u = 1^402^20^35$. Now, define $B_k(n)$ to be the set of binary words of length $n$ with exactly $k$ 1's and at most 1 ascent, and set $b_k(n) = |B_k(n)|$. As an example, 
\[B_4(6) = \{1^40^2, 1^3010, 1^30^21, 1^201^20, 1^20^21^2, 101^30, 10^21^3, 0^21^4, 01^40\},\] so that $b_4(6) = 9$, which happens to be exactly $R_{6,4}$. We are now ready to give our first combinatorial interpretation of Rascal numbers. 

% theorem 1
\begin{theorem} For any $n,k \in \Z$, $R_{n,k} = b_k(n)$.
\label{ras-bin}
\end{theorem}

\begin{proof}  First we see that $B_k(n)$ is the empty set whenever $k > n$ or either of $n$ or $k$ is negative, and so $b_k(n) = 0 = R_{n,k}$ in those cases. Next, $b_0(0) = |B_0(0)| = |\{\epsilon\}| = 1 = R_{0,0}$; similarly, $B_0(1) = \{0\}$ and $B_1(1) = \{1\}$, so that $b_0(1) = 1 = R_{1,0}$ and $b_1(1) = 1 = R_{1,1}$. Therefore, the initial conditions for $b_k(n)$ and $R_{n,k}$ are identical.

Now suppose $0 \leq k \leq n$ and $n \geq 2$. Then any $b \in B_k(n)$ has one of following two forms: 
    \begin{enumerate}
        \item[i.] $1^{k} \; 0^{n-k}$ if $b$ has 0 ascents, or
        \item[ii.] $1^{k-i} \; 0^{j} \; 1^{i} \; 0^{n-k-j}$ with $0 < i \leq k$ and $0 < j \leq n-k$ if $b$ has 1 ascent.
    \end{enumerate} Accordingly, there is only one such word in $B_k(n)$ with 0 ascents.  To construct the words in $B_k(n)$ with exactly 1 ascent, we simply choose allowed values for $i,j$. We can see that there are $k$ choices for $i$, namely $1, 2, \ldots, k$, and similarly, there are $n-k$ choices for $j$. Thus, there are $k(n-k)$ elements of $B_k(n)$with exactly 1 ascent, giving $b_k(n) = 1 + k(n-k) = R_{n,k}$.
\end{proof}

We now prove many other facts about Rascal numbers using this combinatorial interpretation, beginning with showing that these numbers are unimodal and symmetric. To see unimodality, we need only notice that $b_k(n)$ is quadratic in $k$, and thus each row of the Rascal Triangle consists of a unimodal sequence. To show symmetry, we first give a couple of standard definitions for words. Given the binary word $b = b_1b_2\cdots b_n$, define the \emph{reverse} of $b$ to be $b^r = b_n \cdots b_1b_2$, and the \emph{complement} of $b$ to be $b^c = b_1'b_2' \cdots b_n'$, where $b_i' = 0$ if $b_i = 1$ and $b_i' = 1$ if $b_i = 0$. Second, notice that the reverse and complement operations are their own inverses, that is, $b = (b^r)^r = (b^c)^c$.

% theorem 2
\begin{theorem} \label{sym}
For $n,k \in \Z$ with $0 \leq k \leq n$, $R_{n,k} = R_{n,n-k}$.
\end{theorem}

\begin{proof} 
Given $b =1^v0^u1^t0^s \in B_k(n)$, define the function $f$ by $f(b) = (b^r)^c$. Notice that $b^r = 0^s1^t0^u1^v$ still has length $n$ and $k$ 1's. Then $(b^r)^c = 1^s0^t1^u0^v$ has length $n$ and $n-k$ 1's. Accordingly, $f$ is a map from $B_k(n)$ into $B_{n-k}(n)$. However, since $f$ is the composition of two invertible functions, it is itself invertible, and thus a bijection, giving that $|B_k(n)| = |B_{n-k}(n)|$, as desired.
\end{proof}

Next we turn our attention to Fleron's recursion from Equation (\ref{fleron}), which we prove using the Principle of Inclusion/Exclusion, and then we subsequently prove a generalization of Equation (\ref{orec}). Both proofs rely heavily on the following lemma.

\begin{lemma} \label{interp}
    For $n,k,\ell,u \in \Z$ with $0 \le \ell, u$ and $\ell < k < n - u$, $R_{n-u-\ell,k-\ell}$ counts the number of elements in $B_{k}(n)$ which start with at least $\ell$ 1's and end with at least $u$ 0's.
    % \begin{enumerate}
    %     \item[(1)] $R_{n-u,k}$ counts the number elements in $B_{k}(n)$ which end with at least $u$ 0's.
    %     \item[(2)] $R_{n-\ell,k-\ell}$ counts the number of elements in $B_{k}(n)$ which start with at least $\ell$ 1's.
    %     \item[(3)] $R_{n-u-\ell,k-\ell}$ counts the number of elements in $B_{k}(n)$ which start with at least $\ell$ 1's and end with at least $u$ 0's.
    % \end{enumerate}
\end{lemma}
\begin{proof}
    % \item[(1)] Given any element of $B_{k}(n)$ ending in at least $u$ 0's, we can remove $u$ trailing 0's to get an element in $B_{k}(n-u)$.
    % Since this process is invertible, we have a bijection and thus our desired result.

    % \item[(2)] Given any element of $B_{k}(n)$ starting with at least $\ell$ 1's, we can remove $\ell$ leading 1's to get an element in $B_{k-\ell}(n-\ell)$.
    % Since this process is invertible, we have a bijection and thus our desired result.

    % \item[(3)] 
    Given any element of $B_{k}(n)$ that starts with at least $\ell$ 1's and ends with at least $u$ 0's, we can remove $\ell$ leading 1's and $u$ trailing 0's to get an element of $B_{k-\ell}(n-\ell-u)$.
    Since this process is invertible, we have a bijection, and thus we obtain our desired result.
\end{proof}

% \begin{lemma}\label{end0}
% The number of words in $B_k(n)$ that end in 0 is $R_{n-1,k}$.
% \end{lemma}

% \begin{proof}
% Suppose $b \in B_k(n)$ ends in 0. Then by removing this 0 we end up with a word in $B_k(n-1)$. However, this operation is reversible: take any $b' \in B_k(n-1)$ and add a 0 to its end to create the word $b''$. Then $b''$ and $b'$ have the same number of 1's and the same number of ascents, so that $b'' \in B_k(n)$ and ends in 0. 
% \end{proof}

% theorem 4
\begin{theorem}[Fleron's recursion] \label{rec1} Given $R_{0,0} = R_{1,0} = R_{1,1} = 1$ and $R_{n,k} = 0$ if $k > n$ or if $n$ or $k$ is negative, 
\[R_{n,k} = R_{n-1,k} + R_{n-1,k-1} - R_{n-2,k-1} + 1.\] 
\end{theorem}

\begin{proof}
We have that the left side of this recursion, $R_{n,k}$, counts the number of words in $B_k(n)$, and all of these words either end in a 0 or end in a 1. By Lemma \ref{interp}, there are exactly $R_{n-1,k}$ of these that end in 0. Similarly, we may add a 1 to the end of any binary word of length $n-1$ with exactly $k-1$ 1's and at most 1 ascent, and we will obtain all of the words counted ending in 1 counted by $R_{n,k}$: there are $R_{n-1,k-1}$ such words. However, we have possibly overcounted in this latter case, because any word of length $n-1$ that was of the form $1^{k-1-u}0^t1^u0^v$ with $t, u,v > 0$ will have 2 ascents when we add a 1 to the end. Reapplying our lemma, there are $R_{n-2,k-1} - 1$ such words, because $1^{k-1}0^{n-k+1}$ is counted by $R_{n-2,k-1}$, but this is the only such word which is not of the form $1^{k-1-u}0^t1^u0^v$ with $t, u,v > 0$. Thus, the number of words in $B_k(n)$ ending in 1 is $R_{n-1,k-1} - (R_{n-2,k-1} - 1)$.
\end{proof}

% We next prove a generalization of the original recurrence, Equation (\ref{orec}), by forming bijections to progressively pair off elements, and this process relies on the following interpretations.

To make the following proof a bit more clear and concise, we first introduce some notation.
Throughout the following proof, we use $\ast$ to mean any ``appropriate" binary word.
Formally, for any binary words $w_{0}$ and $w_{2}$, the $\ast$ in $w_{0} \ast w_{2} \in B_{k}(n)$ represents any binary word $w_{1}$ such that $w_{0} w_{1} w_{2} \in B_{k}(n)$.
For example, if we have $1 \ast 0 \in B_2(4)$, then $\ast$ represents any word in the set $\{ 10, 01 \}$.

% theorem 6
\begin{theorem}[Generalization of Equation (\ref{orec})] \label{genrec}
    Let $n,k,\ell,u \in \Z$ with $0 \le \ell, u$ and $\ell \le k \le n-u$.
    Then, we have
\begin{equation}
           R_{n,k} = \frac{R_{n-u,k}R_{n-\ell,k-\ell} + u \ell}{R_{n-u-\ell,k-\ell}}.
           \label{grec}
\end{equation}
\end{theorem}
\begin{proof}
    We will equivalently prove that \[
        u\ell = R_{n,k}R_{n-\ell-u,k-\ell} - R_{n-u,k}R_{n-\ell,k-\ell}.
    \]
    Using Lemma \ref{interp}, we have that \[
        \begin{array}{ccc}
            S_{0} = \left\{ (w, \; 1^{\ell} \ast 0^{u}) \in \left( B_{k}(n) \right)^{2} \right\} & \text{and} &
            T_{0} = \left\{ (\ast \; 0^{u}, \; 1^{\ell} \ast) \in \left( B_{k}(n) \right)^{2} \right\}
        \end{array}
    \]
    are counted by $R_{n,k}R_{n-\ell-u,k-\ell}$ and $R_{n-u,k}R_{n-\ell,k-\ell}$, respectively.
    First, note that elements of the form $(\ast \; 0^{u}, 1^{\ell} \ast 0^{u})$ are in both $S_{0}$ and $T_{0}$.
    So they naturally pair off with one another, which leaves us with \[
        \begin{array}{c}
            S_{1} = \left\{ (\ast \; 10^{x}, \; 1^{\ell} \ast 0^{u}) \in \left( B_{k}(n) \right)^{2} \mid 0 \le x < u \right\} \\ 
            \text{and} \\
            T_{1} = \left\{ (\ast \; 0^{u}, \; 1^{\ell} \ast 10^{x}) \in \left( B_{k}(n) \right)^{2} \mid 0 \le x < u \right\}.
        \end{array}
    \]
    Now we can apply the map $(w,z) \mapsto (z,w)$ on elements in $T_{1}$ of the form $(1^{\ell} \ast 0^{u}, 1^{\ell} \ast 10^{x})$ where $0 \le x < u$.
    This results in all elements of the form $(1^{\ell} \ast 10^{x}, 1^{\ell} \ast 0^{u})$ in $S_{1}$ being paired off with an element of $T_{1}$.
    Therefore we are left with \[
        \begin{array}{c}
            S_{2} = \left\{ (1^{y}0 \ast 10^{x}, \; 1^{\ell} \ast 0^{u}) \in \left( B_{k}(n) \right)^{2} \mid 0 \le x < u, 0 \le y < \ell \right\} \\ 
            \text{and} \\
            T_{2} = \left\{ (1^{y}0 \ast 0^{u}, \; 1^{\ell} \ast 10^{x}) \in \left( B_{k}(n) \right)^{2} \mid 0 \le x < u, 0 \le y < \ell \right\}.
        \end{array}
    \]
    Note that $T_{2}$ does not contain any elements where either component has 0 ascents as those are of the form $1^{k}0^{n-k}$, but $x < u \le n-k$ and $y < \ell \le k$.
    So both components of a pair in $T_{2}$ must have 1 ascent.
    Note that knowing that a binary word has 1 ascent, starts with $a$ 1's, and ends with $b$ 0's, is sufficient to uniquely determine the word.
    Therefore, $1^{a}0 \ast 10^{b}$ must be the word $1^{a}0^{n-k-b}1^{k-a}0^{b}$, and we will exploit this to make the following map easier to read.
    We apply the map \[ 
        (1^{y}0 \ast 10^{a}, 1^{b}0 \ast 10^{x}) \mapsto (1^{y}0 \ast 10^{x}, 1^{b}0 \ast 10^{a}) 
    \] to all elements of $T_2$ to get elements of $S_{2}$---this simply interchanges the number of 0's which each component ends in, and adjusting the number of zeros elsewhere in the word (to ensure the words in each component remain in $B_k(n)$).
    Note that this map is injective as we only need to know $a,b,x,y$ to reconstruct the original pair, and that information can be extracted from the image of the pair.
    Also, since $0 < n-k-a$ and $0 < k-b$, the above map does not map to elements of the form $(1^{y} \ast 0^{x}, \; 1^{k}0^{n-k}) \in S_{2}$ where $0 \le x < u$ and $0 \le y < \ell$.
    Since choosing a value for $x$ and $y$ uniquely determines such an element of $S_{2}$, we only need to count the number of choices for $x,y$.
    There are $u$ choices for $x$ and $\ell$ choices for $y$ so there are exactly $u \ell$ elements of $S_{2}$ which are not in the image of this map.
    Therefore we conclude that \begin{align*}
        u \ell &= |S_{2}| - |T_{2}| \\
        &= |S_{1}| - |T_{1}| \\
        &= |S_{0}| - |T_{0}| \\
        &= R_{n,k}R_{n-u-\ell,k-\ell} - R_{n-u,k}R_{n-\ell, k-\ell}.
    \end{align*}
\end{proof}

The following results provide two other recursive relationships between Rascal numbers.

\begin{theorem}
    For $0 \leq n,m,k$, \begin{equation}
        R_{n+m,k} = R_{n,k} + R_{m+k,k} - 1.
    \label{recm}
\end{equation}
\end{theorem}
\begin{proof} 
    By Theorem \ref{ras-bin}, $R_{n+m,k} = b_{k}(n+m)$, so we need to show that the left hand side of Equation (\ref{recm}) counts this as well.

By Lemma \ref{interp}, $R_{n,k}$ is the number of $w \in B_k(n+m)$ which end with at least $m$ 0's.
    So, we only need to count the number of $w \in B_k(n+m)$ which end in strictly less than $m$ 0's. 
    Since $w$ ends in less than $m$ 0's and there are $n+m-k$ total 0's in $w$, the first consecutive string of 0's must contain more than $n-k$ 0's. So, we can remove the first $n-k$ 0's in the string to get a word with $k$ 1's and $n+m-k-(n-k) = m$ 0's and one ascent.
    This new word is an element of $B_k(m+k) \setminus \left\{ 1^k0^m \right\}$.
    Furthermore, since this process is reversible, there are exactly $|B_k(m+k) \setminus \left\{ 1^k0^m \right\}| = R_{m+k.k} - 1$ such words.
    Summing these two cases, we get the right hand side of Equation (\ref{recm}), as desired.
\end{proof}

\begin{theorem}
    For $n,k \in \Z$ with $0 < k < n$, $kR_{n-1,k-1}-1 = (k-1)R_{n,k}.$
\end{theorem}
\begin{proof}
    Let $S$ be the set of $w \in B_{k}(n)$ where $w$ begins with a 1 and we've circle one of the 1's in $w$, and let $T$ be the set of $w \in B_{k}(n)$ where we've circled a 1 in $w$ which is not the first in $w$.
    By  Lemma \ref{interp}, $R_{n-1,k-1}$ counts $w \in B_{k}(n)$ which begin with a 1, we have
     \[
        |S| = kR_{n-1,k-1}  \text{ and } |T| = (k-1)R_{n,k},
    \]
    were $k$ and $k-1$ account for circling the 1's, respectively.
    So it suffices to construct an injective map from $T$ to $S$ whose image has cardinality $|S|-1$.
    The map, $f$, is as follows.
    Take a $w \in T$: \begin{enumerate}
        \item[(1)] If $w$ starts with a 1 then set $f(w) = w \in S$.
        \item[(2)] If $w$ starts with a 0 then all 1's in $w$ are located in a consecutive string.
        Spilt this consecutive string of 1's before the circled 1 then move the right half to the beginning of $w$.
        Since this new word, $z$, is still in $B_{k}(n)$, has a circled 1, and begins with a 1, it is an element of $S$.
        So set $f(w) = z$.
    \end{enumerate}
    Note that case (1) maps $w$ to a word whose first 1 is not circled while case (2) maps $w$ to a word whose first 1 is circled, so there is no overlap in the image of cases (1) and (2).
    Since both case (1) and (2) are easily inverted, $f$ is injective.
    Moreover, the element $b = 1^{k}0^{n-k}$ where the first 1 is circled is in $S$ but is not mapped to by $f$.
    To see this, we note that if there exists a $w \in T$ such that $f(w) = b$ then $w$ must fall into case (2) since the first 1 of $b$ is circled.
    But, since the first 1 of $w$ is not circled, applying $f$ to $w$ leaves an ascent but $\operatorname{asc}(b) = 0$.
    Since $f$ is injective and only 1 element of $S$ is not in the image of $f$, we conclude that $|T| = |\text{the image of } f| = |S|-1$, that is $(k-1)R_{n,k} = kR_{n-l,k-1} - 1$.
\end{proof}

%%%%%%%%%%%%%%%
%%%%%%%%%%%%%%%
%%%%%%%%%%%%%%%
%%%%%%%%%%%%%%%
%%%%%%%%%%%%%%%

\section{Some Identities} \label{identities}

Since our primary combinatorial interpretation of the Rascal numbers involves a subset of binary words of length $n$ with $k$ 1's and $n$ choose $k$ counts the total number of binary words with length $n$ and $k$ 1's, it seems natural to look for binomial-type identities involving $R_{n,k}$. To that end, we begin by proving various row and column sum identities for $R_{n,k}$.

\begin{theorem}[Row Sum] \label{rowsum}
For $0 \leq n \in \Z$, \begin{equation}
    \sum_{k=0}^{n} R_{n,k} = \binom{n+1}{3} + n+1.
\label{binomrowsum}
\end{equation}
\end{theorem}
\begin{proof}
    Setting \[
        B(n) = \bigcup_{k=0}^{n} B_k(n),
    \]
     we first note that the left hand side of Equation (\ref{binomrowsum}) counts $B(n)$.
    
Now pick $w \in B(n)$. If $\operatorname{asc}(w) = 0$, then $w=1^{k}0^{n-k}$ for some $0 \le k \le n$.
        Since there are $n+1$ choices of $k$, there are exactly $n+1$ such $w$. Otherwise $\operatorname{asc}(w) = 1$, so that $w=1^{x}0^{z}1^{y}0^{n-x-y-z}$ where $0 < z,y$ and $0 \le x$ with $x+y+z \le n$.
        Since choosing appropriate $x,y,z$ uniquely determines $w$, it suffices to count solutions to $x+y+z \le n$.
        If $x+y+z \le n$, then there exists some $t \in \Z$ with $t \ge 0$ such that $x+y+z+t = n$.
        Note that solutions to this equation are in one-to-one correspondence with positive integer solutions to $x'+y+z+t' = n+2$.
        There are exactly \[
            \binom{n+2-1}{4-1} = \binom{n+1}{3}
        \]
        such solutions, so there are the same number of $w \in B(n)$ with $\operatorname{asc}(w) = 1$.

    Summing these two cases, we find that the right hand side of Equation (\ref{binomrowsum})  also counts $B(n)$, giving the desired equality.
\end{proof}

\begin{theorem}[Column Sum] \label{colsum}
For $0 \leq k, r \in \Z$, \begin{equation}
    \sum_{i=0}^{r} R_{k+i,k} = k\binom{r+1}{2} + r + 1.
\label{binomcolsum}
\end{equation}
\end{theorem}
\begin{proof}
    Note that the left hand side of Equation (\ref{binomcolsum}) counts $B = \bigcup_{i=0}^{r} B_k(k+i)$, and any word in $B$ must have a length of at least $k$ but at most $k+r$.
    So, for each possible length, there is exactly one word in $B$ with 0 ascents.
    Therefore, there are $r+1$ such words in $B$.
    Now, we only need to count $w \in B$ such that $\operatorname{asc}(w)=1$.
    Such $w$ must have the form: $1^{k-x}0^{y+1}1^{x}0^{z}$ where $1 \le x \le k$, $0 \le y, z$ and $(y+1)+z \le r$.
    We note that a choice of $x,y,z$ uniquely determines $w$, and vice versa, so such 3-tuples are in one-to-one correspondence with $w \in B$ such that $\operatorname{asc}(w)=1$.
    Therefore, it suffices to count all such 3-tuples.
    If $(y+1)+z \le r$, then there exists $t \in \Z$ with $0 \le t$ and $(y+1)+z+t=r$, or equivalently, $y+z+t = r-1$.
    There are exactly \[
        \binom{(r-1)+3-1}{3-1} = \binom{r+1}{2}
    \]
    non-negative integer solutions to this equation, and since there are exactly $k$ choices for $x$, we have that there are exactly \[ k\binom{r+1}{2} \] such 3-tuples.
    Summing the two cases together, we find that \[
        \sum_{i=0}^{r} R_{k+i,k} = |B| = k\binom{r+1}{2} + r+1
    \]
    as desired.
\end{proof}

\begin{theorem}[Binomial-Weighted Row Sum]
    For $0 \leq n \in \Z$, \begin{equation}
        \sum_{k=0}^{n} \binom{n}{k}R_{n,k} =  2^{n-2}\binom{n}{2} + 2^{n}.
\label{wrowsum}
    \end{equation}
\end{theorem}
\begin{proof}
    For any $0 \leq k \leq n$, 
    let $T_k$ be the set of all pairs $(w_{1}, w_{2})$ where $w_{1}$ is a binary word of length $n$ with exactly $k$ 1's and $w_{2} \in B_{k}(n)$, and set $T = \bigcup_{k=0}^n T_k$. Then $|T_k| = \binom{n}{k}R_{n,k}$, and $T_i \cap T_j = \emptyset$ whenever $i \neq j$. Accordingly, the left hand side of Equation (\ref{wrowsum}) is $|T|$.
    Now for any $(w_{1},w_{2}) \in T$, either $\operatorname{asc}(w_{2}) = 0$ or $\operatorname{asc}(w_{2}) = 1$.
    \begin{enumerate}
        \item[(1)] There are $2^n$ total binary words of length $n$, and for each $0 \leq k \leq n$ there is exactly one word with 0 ascents. So, there are $2^n$ pairs $(w_1,w_2) \in T$ such that $\operatorname{asc}(w_{2}) = 0$.

        \item[(2)] If $\operatorname{asc}(w_{2}) = 1$, then we can construct the pair $(w_{1}, w_{2})$ as follows.
        Note that $w_{1},w_{2}$ must both have at least one 0 and one 1 since $\operatorname{asc}(w_{2})=1$.
        Choose $\{ i,j \} \subseteq [n]$ with $i < j$ and place a 0 at the $i$-th index of $w_{1}$ and a 1 at the $j$-th index of $w_{1}$.
        Then fill in the remaining $n-2$ spots of $w_{1}$.
        Let $k$ be the number of 1's in $w_{1}$, $r$ be the number of 1's with index strictly less that $j$, and $s$ be the number of 0's with index strictly less than $i$.
        Then we set $w_{2} = 1^{k-r-1}0^{s+1}1^{r+1}0^{n-k-s-1}$.
        Note that we can recover $r$ and $s$ from $w_{2}$, which gives that the $i$-th bit of $w_{1}$ is the 0 which is proceeded by exactly $s$ 0's and the $j$-th bit is the 1 which is proceeded by exactly $r$ 1's.
        Therefore, we can determine $\left\{ i,j \right\}$ and the binary word of length $n-2$ from the pair $(w_{1},w_{2})$ so this construction is invertible.
        Thus, there are $2^{n-2}\binom{n}{2}$ such pairs.
    \end{enumerate}
    Summing cases (1) and (2), we conclude that the right hand side of Equation (\ref{wrowsum}) is also $|T|$, giving the desired equality.
\end{proof}

The next identity was given in Ashfaque \cite{Ash} (proved algebraically), and so we provide a combinatorial proof here. This identity also gives a combinatorial interpretation of the sequence \href{https://oeis.org/A051744}{A051744} \cite{OEIS}.

\begin{theorem} \label{trianglesum}
    For $n \in \Z$ with $n \ge 2$, \begin{equation}
        \sum_{i=1}^{n} \sum_{j=1}^{i-1} R_{i,j} = \binom{n+2}{4} + \binom{n}{2}.
    \label{ashf}
\end{equation}
\end{theorem}
\begin{proof}
    Define \[
        T = \bigcup_{i=1}^{n} \left( \bigcup_{j=1}^{i-1} B_j(i) \right),
    \]
    that is, $T$ is the set of all binary words of length at most $n$ with 0 or 1 ascents and at least one 0 and one 1.
    Note that the left hand side of Equation (\ref{ashf}) is $|T|$.
    Let $w \in T$.
    \begin{enumerate}
        \item[(1)] If $\operatorname{asc}(w) = 0$, then $w$ takes the form $1^{x}0^{y}$ where $x,y > 0$ and $x + y \le n$.
        We note that the choice of $x,y$ uniquely determines $w$ so it suffices to count such pairs $(x,y)$.
        Since $x + y \le n$, there exists a $t \in \Z$ such that $t \ge 0$ and $x+y+t = n$.
        Solutions to this equation are in one-to-one correspondence with positive integer solutions to $x+y+t' = n+1$, and there are exactly \[
            \binom{n+1-1}{3-1} = \binom{n}{2}
        \]
        such solutions.

        \item[(2)] Alternatively, we could have $\operatorname{asc}(w) = 1$. Then $w$ would take the form $1^{x}0^{y}1^{z}0^{t}$ such that $y,z > 0$, $x,t \ge 0$, and $x+y+z+t \le n$.
        Since choosing such $x,y,z,t$ determines uniquely $w$, it suffices to count all such tuples $(x,y,z,t)$.
        Since $x+y+z+t \le n$, there exists some $a \in \Z$ such that $a \ge 0$ and $x+y+z+t+a = n$.
        Solutions to this equation are in one-to-one correspondence with positive integer solutions to $x'+y+z+t'+a' = n + 3$, and there are exactly \[
            \binom{n+3-1}{5-1} = \binom{n+2}{4}
        \]
        such solutions.
    \end{enumerate}
    Thus, summing cases (1) and (2) gives that $|T|$ is also the right hand side of Equation (\ref{ashf}).
\end{proof}

The next identity involves alternating sums, and we construct involutions to pair off elements of opposite sign.
To begin, for a function $f\text{: } X \to X$, we denote the set of all fixed points of $f$ as $\operatorname{fix}(f)$.

\begin{theorem}
    For $n,k,r \in \Z$ with $0 \le k \le n$ and $r \geq 2$, \begin{equation}
        \sum_{j=0}^{r}(-1)^{r-j}\binom{r}{j}R_{n+j,k} = 0. 
        \label{altbin}
    \end{equation}
\end{theorem}
\begin{proof} Let $r \ge 2$, with the power set of $[r]$ denoted by $\mathcal{P}([r])$.
    Set \[ T = \left\{ (S,w) \in \mathcal{P}([r]) \times B_{k}(n+r) \ | \  w \text{ ends in at least $r-|S|$ 0's} \right\}, \] and for any set $S$ define \[
        S \oplus x = \begin{cases}
            S \setminus \left\{ x \right\}, & \text{if $x \in S$} \\
            S \cup \left\{ x \right\}, & \text{if $x \notin S$}
        \end{cases}.
    \]
    Next, define $\operatorname{wt}\text{: } T \to \left\{ \pm 1 \right\}$ by $(S,w) \mapsto (-1)^{r-|S|}$ for any $(S,w) \in T$. Then \begin{equation}\sum_{j=0}^{r}(-1)^{r-j}\binom{r}{j}R_{n+j,k} = 
        \sum_{(S,w) \in T} \operatorname{wt}(S,w).
        \label{xxx}
    \end{equation}
    Now, define $I_{1}\text{: } T \to T$ as follows. 
    \begin{enumerate}
        \item[(1)] If $w$ ends in at least $r$ 0's then set \[
            I_{1}(S,w) = (S \oplus r, w).
        \]
        
        \item[(2)] This leaves the case where $w$ ends in fewer than $r$ 0's so $w = 1^{k-x}0^{n+r-k-y}1^{x}0^{y}$ where $1 \le x \le k$ and $r-|S| \le y < r$.
        Then we set \[
            I_{1}(S,w) = \begin{cases}
                (S \oplus r, w), & \text{if $r-|S| < y$} \\
                (S \oplus r, w), & \text{if $r-|S|=y$ and $r \notin S$} \\
                (S,w), & \text{if $r-|S|=y$ and $r \in S$}
            \end{cases}.
        \]
    \end{enumerate}
    Since $\oplus$ is an involution, $I_{1}$ is itself a sign-reversing, weight-preserving involution, and so Equation (\ref{xxx}) becomes \[
       \sum_{j=0}^{r}(-1)^{r-j}\binom{r}{j}R_{n+j,k} = \sum_{(S,w) \in \operatorname{fix}(I_{1})} \operatorname{wt}(S,w).
    \]
    By construction of $I_{1}$, we have \[
        \operatorname{fix}(I_{1}) = \left\{ (S,w) \in T : w \text{ ends in exactly $r-|S|$ 0's and $r \in S$} \right\}.
    \]
    Now define $I_{2}\text{: }  \operatorname{fix}(I_{1}) \to \operatorname{fix}(I_{1})$ as follows.
    For $(S,w) \in \operatorname{fix}(I_{1})$ where $w=1^{k-x}0^{n+|S|-k}1^{x}0^{r-|S|}$, set \[
        I_{2}(S,w) = \begin{cases}
            (S \oplus 1, 1^{k-x}0^{n+|S|+1-k}1^{x}0^{r-|S|-1}), & \text{if $1 \notin S$} \\
            (S \oplus 1, 1^{k-x}0^{n+|S|-1-k}1^{x}0^{r-|S|+1}), & \text{if $1 \in S$}
        \end{cases}.
    \]
    We note that $I_{2}$ is a sign-reversing, weight-preserving involution with no fixed points.
    Thus our sum is 0 when $r \geq 2$, as desired.
\end{proof}

Another alternating sum identity is given here, but this result follows directly from Theorem \ref{genaltrowsum} in the last section, and so we omit a proof here.

\begin{theorem}[Alternating Row Sum] \label{altrowsum1}
    For $0 \leq n \in \Z$, \[
        \sum_{k=0}^{n} (-1)^{k}R_{n,k} = \begin{cases}
            0, & \text{if $n \equiv 1 \pmod{2}$} \\
            1 - \frac{n}{2}, & \text{if $n \equiv 0 \pmod{2}$}
        \end{cases}.
    \]
\end{theorem}

We conclude this section with two product formulas involving Rascal numbers. 
 First, for $0 \leq k, n \in \Z$, we define the \emph{falling factorial}, $(n) \hspace{-1mm} \downarrow_{k}$, to be 0 if $k > n$, 1 if $k=0$, and otherwise
 \[
    (n) \hspace{-1mm} \downarrow_{k}=     \displaystyle \prod_{i=1}^{k}(n-i+1). 
\]

\begin{theorem} \label{relationinspo}
    For $n,m \in \Z$ with $1 \le m \le n$, 
    \begin{equation}       
        \prod_{k=1}^{m} \left( R_{n,k} - 1 \right) = \sum_{S \subseteq [m]} (-1)^{m-|S|} \prod_{i \in S} R_{n,i} = m! (n-1) \hspace*{-1mm} \downarrow_{m}.
        \label{prodform}
       \end{equation}
\end{theorem}
\begin{proof}
    Notice that $R_{n,k} - 1$ is the number of binary words with length $n$ and $k$ 1's which have 1 ascent.
    Therefore, the product on the left enumerates $|T|$, where 
    \[T = \{ (w_{1}, \ldots, w_{m}) \; \mid \; w_{i} \in B_{i}(n), \; \operatorname{asc}(w_{i}) = 1 \}.\]
    For any $(w_{1}, \ldots, w_{m}) \in T$, $w_{i} = 1^{i-x}0^{y}1^{x}0^{n-i-y}$ where $1 \le x \le i$ and $1 \le y \le n-i$.
    So choosing values for $x$ and $y$ uniquely determines $w_{i}$.
    Since $1 \le x \le i$ and $1 \le y \le n-i$, there are $i(n-i)$ such $w_{i}$'s, which implies that there are $m!(n-1) \hspace{-1mm} \downarrow_{m}$ such $(w_{1}, \ldots, w_{m})$'s.
    Alternatively, we can find $|T|$ using the Principle of Inclusion/Exclusion by considering the set $V = \{ (w_{1}, \ldots, w_{m}) \mid w_{i} \in B_{i}(n) \}$. Then, setting $A_{i} = \{ (w_{1}, \ldots, w_{m}) \in V \mid \operatorname{asc}(w_{i}) = 0 \}$ gives that $|T| = |A_{1}^{c} \cap \ldots \cap A_{m}^{c}|$ where $A_{i}^{c}$ denotes the complement of $A_{i}$ in $V$.
    Since there is exactly one word in $B_{i}(n)$ with 0 ascents, \[
        |A_{i_{1}} \cap \cdots \cap A_{i_{k}}| = \prod_{j \in [m] \setminus \{ i_{1}, \ldots, i_{k} \}} R_{n,j}.
    \]
Thus, we get \[
        |T| = \sum_{k=0}^{m} (-1)^{k} \sum_{\stackrel{S \subseteq [m]}{|S|=m-k}} \prod_{i \in S} R_{n,i} = \sum_{S \subseteq [m]} (-1)^{m-|S|} \prod_{i \in S} R_{n,i}.
    \]
\end{proof}

With the above result established, another relationship between $n$ choose $k$ and $R_{n,k}$ becomes apparent. While the following result is a direct consequence of simply dividing both sides of Equation (\ref{prodform}) by $(m!)^2$, we provide a combinatorial proof.

\begin{corollary} \label{relation}
    For $n,m \in \Z$ with $1 \le m \le n$, \begin{equation}
        \frac{1}{(m!)^{2}}\prod_{k=1}^{m} \left( R_{n,k} - 1 \right) = \frac{1}{(m!)^{2}} \sum_{S \subseteq [m]} (-1)^{m-|S|} \prod_{i \in S} R_{n,i} = \binom{n-1}{m}.
        \label{T2}
        \end{equation}
\end{corollary}
\begin{proof}
    Let $T$ be the set defined in Theorem \ref{relationinspo}, so that Equation (\ref{T2}) becomes \[
        \frac{1}{(m!)^{2}}|T| = \frac{1}{(m!)^{2}} |T| = \binom{n-1}{m}.
        \]
        Now define $\sim_{1}$ on $B_{i}(n) \setminus \{ 1^{i}0^{n-i} \}$ by $w_{1} \sim_{1} w_{2}$ if and only if $w_{1} = 1^{i-x}0^{y}1^{x}0^{n-i-y}$ where $1 \le x \le i$, $1 \le y \le n-i$ and $w_{2} = 1^{i-z}0^{y}1^{z}0^{n-i-y}$ where $1 \le z \le n-i$.
    In essence, under $\sim_{1}$ we only care about the relative position of the 0's in the words without regard for the position of the 1's. 
    Next, we define $\sim_{2}$ on $T$ by $(w_{1}, \ldots, w_{m}) \sim_{2} (w_{1}', \ldots, w_{m}')$ if and only if $w_{i} \sim_{1} w_{i}'$ for all $1 \le i \le m$.
    Since, for a fixed $y$, there are $i$ choices of $x$, we have that $\left| \left[w_{i} \right]_{\sim_{1}} \right| = i$ for any $w_{i} \in B_{i}(n) \setminus \{ 1^{i}0^{n-i} \}$.
    So 
    \[ \left| \left[(w_{1}, \ldots, w_{m}) \right]_{\sim_{2}} \right| = \prod_{i=1}^{m} \left| \left[ w_{i} \right]_{\sim_1} \right| = m! \]
    for any $(w_{1}, \ldots, w_{m}) \in T$, which implies $|T/\sim_{2}| = |T|/m!$.
    Note that to uniquely determine an equivalence class of $\sim_{2}$, we only need a $y_{i}$ where $w_{i} = 1^{\ast}0^{y_{i}}1^{\ast}0^{n-i-y_{i}}$ with $1 \le y_{i} \le n-i$ for each $i \in [m]$.
    Moreover, given an equivalence class of $\sim_{2}$, we can uniquely determine a $m$-tuple $(y_{1}, \ldots, y_{m})$ where $1 \le y_{i} \le n-i$ by taking $y_{i}$ to be the number of 0's preceding the first ascent.
    Therefore, elements of $T/\sim_{2}$ are in one-to-one correspondence with $m$-tuples of the form $(y_{1}, \ldots, y_{m})$, where $1 \le y_{i} \le n-i$.
    Let $f$ be a map which takes such an $m$-tuple, $(y_{1}, \ldots, y_{m})$, to an $m$-permutation of $[n-1]$, $a_{1} \ldots a_{m}$, defined as follows.
    Set $a_{1} = y_{1}$ and $S_{1} = [n-1] - \{ y_{1} \}$.
    Then set $a_{i}$ to the $y_{i}$-th smallest element of $S_{i-1}$ for $1 < i \le m$.
    This map can be inverted setting $y_{1} = a_{1}$ and $S_{1} = [n-1] - \{ a_{1} \}$.
    Then set $y_{i} = j$ where $a_{i}$ is the $j$-th smallest element of $S_{i-1}$ for $1 < i \le m$.
    Thus, $T/\sim_{2}$ is in one-to-one correspondence with $m$-permutations of $[n-1]$.
    Then we can remove the ordering of the $m$-permutations giving us $m$-element subsets of $[n-1]$ which implies that \[
        \binom{n-1}{m} = \frac{|T/\sim_{2}|}{m!} = \dfrac{1}{(m!)^2}|T|.
    \]
\end{proof}

%%%%%%%%
%%%%%%%%
%%%%%%%%
%%%%%%%%
%%%%%%%%
%%%%%%%%

\section{Ascent Sequences \& Pattern Avoidance}

Going back to ascents in words, notice that only the relative values of the integers that comprise a word are important when computing the ascent number of a word, and so it is common to define the \emph{reduction of a word}, $w = w_1w_2\cdots w_n$, to be the word obtained by replacing the $i$-th smallest letter of $w$ with the letter $i-1$, denoted by red($w$). For example, red$(2151159858) = 1020024323$, and we see that for any $w$, asc$(w) = $ asc$($red$(w))$. We now define a \emph{pattern} to be any word which is its own reduction; for instance $w = 01259$ is not a pattern, while $p = 0210$ is a pattern. We then say that the word $w$ \emph{contains} the pattern $p = p_1p_2\cdots p_k$ if there exists a sequence $1 \leq i_1 < i_2 < \cdots < i_k \leq n$ such that red$(w_{i_1} w_{i_2} \cdots w_{i_k}) = p$, and otherwise we say that \emph{w avoids p}. Considering the word $u = \underline{5}579\underline{0}2\underline{4}$, $u$ contains the pattern $p = 201$, which may be realized by the underlined letters of $u$ (and by many other subwords as well), while the pattern $q = 210$ is avoided by $u$, as indeed there is no 3-letter subsequence of $u$ in which the letters appear in decreasing order. A special type of pattern that we will use in this section is a \emph{restricted growth function}, or \emph{RGF}, which is any pattern in which the first occurrence of the letter $k$ must be preceded by the letter $k-1$ for any $k \geq 1$, and hence, the first occurrence of $k$ must actually be preceded by all of $0, 1, \ldots, k-1$. For example, $p = 0012332041$ is an RGF, while $q = 210$ is not. Finally, we define an \emph{ascent sequence} to be a word $w = w_1w_2\cdots w_n$ such that 
\begin{enumerate}
\item $w_1 = 0$, and
\item for $1 < i \leq n$, $w_i \leq \text{asc}(w_1\cdots w_{i-1}) + 1$.
\end{enumerate}
For example, 012345 and 012000 are both ascent sequences, whereas 001162 is not, since $6 = w_5 > \text{asc}(w_1w_2w_3w_4) = \text{asc}(0011) = 1$. Bousquet-M\'elou et al \cite{B-M} were the first to define ascent sequences. Subsequently, Duncan and Steingr\'imsson \cite{DS} were the first to study pattern avoidance in ascent sequences, specifically considering how many ascent sequences of length $n$ avoid various given patterns. A more recent paper by Baxter and Pudwell \cite{BP} asks how many ascent sequences of length $n$ simultaneously avoid a given collection of patterns. To this end, let us define $\A_{\mathcal{P}}(n)$ to be the set of ascent sequences of length $n$ that avoid the collection of patterns $\mathcal{P}$, where $\mathcal{P}$ may contain only a single pattern. For example, there are 15 ascent sequences of length 4:
\begin{center}
0000, 0001, 0010, 0100, 0011, 0101, 0110, 0111, \\ 
0012, 0102, 0112, 0120, 0121, 0122, 0123.
\end{center} 
By inspection, $\A_{\{001,210\}}(4) = \{0000, 0100, 0110, 0111, 0120, 0121, 0122, 0123\}$.
If we further restrict the elements of $\AP{n}$ by defining $\APK{n}{k}$ to be the set of all $w \in \AP{n}$ where $\operatorname{asc}(w)=k$, we have the following: \begin{align*}
    \APK{4}{0} &= \left\{ 0000 \right\}, \\
    \APK{4}{1} &= \left\{ 0100, 0110, 0111 \right\}, \\
    \APK{4}{2} &= \left\{ 0120, 0121, 0122 \right\}, \\
    \APK{4}{3} &= \left\{ 0123 \right\}, \\
    \APK{4}{4} &= \emptyset.
\end{align*}
From the above example, we can conjecture and prove a nice form for words in $\APK{n}{k}$.

\begin{lemma} \label{ascform}
    For $n,k \in \Z$ with $0 \le k \le n$, all $w \in \APK{n}{k}$ have the form: \[
        01 \ldots k^{n-k} \hspace{5mm} \text{or} \hspace{5mm} 01 \ldots k^{y}x^{n-k-y}
    \]
    where $1 \le y < n-k$ and $0 \le x < k$.
\end{lemma}
\begin{proof}
    Let $w = w_{1} \ldots w_{n} \in \APK{n}{k}$ and define $M$ to be the largest letter in $w$.
    Note that our statement is trivially true when $k=0$, so we may consider when $k>0$ and, therefore, there are at least 2 distinct letters in $w$.
    Since $w$ avoids 001 and 001 is a subpattern of 01012, $w$ is an RGF by Lemma 2.4 of Duncan and Steingr\'imsson \cite{DS}.
    So, $0,1, \ldots, M$ must all appear in $w$ and their first occurences must appear in this order.
    Suppose that $a < M$ occurs more than once before $M$, then we'd have the subsequence $aaM$ which forms a 001 pattern.
    Therefore, $w = 01 \ldots (M-1)M w'$ for some word $w'$.
    Suppose that $w'$ contains letters $a,b$ with $M > a > b$.
    If $a$ occurs first in $w'$, then we have the subsequence $Mab$ which forms a 210 pattern.
    If $b$ occurs first in $w'$, then, since $b$ must occur in $01 \ldots (M-1)M$, we have the subsequence $bba$ which forms a 001 pattern.
    Thus, $w'$ can contain at most one letter strictly less than $M$ which implies that $w'$ must be weakly decreasing.
    Therefore, $w'$ has no ascents, which gives $k = \operatorname{asc}(w) = \operatorname{asc}(01 \ldots Mw') = M$.
    So, we have that \[
        w = 01 \ldots k^{n-k}
    \]
    if no letter other than $M$ appears in $w'$, and otherwise \[
        w = 01 \ldots k^{y}x^{n-k-y}
    \]
    where $x$ is the letter strictly less than $k$ which occurs in $w'$.
\end{proof}

We are now in a position to prove that $\APK{n+1}{k}$ is enumerated by $R_{n,k}$ by constructing a bijection between $\APK{n+1}{k}$ and $B_{k}(n)$.

\begin{theorem} \label{binascbijection}
    For $n,k \in \Z$, we have \[
        |\APK{n+1}{k}| = R_{n,k}.
    \]
\end{theorem}
\begin{proof}
    When $n,k < 0$, we have that $\APK{n+1}{k} = \emptyset$, so $|\APK{n+1}{k}|=0=R_{n,k}$ in such cases.
    Also, since a word of length $n+1$ can have at most $n$ ascents, $\APK{n+1}{k} = \emptyset$ when $k > n$, so $|\APK{n+1}{k}|=0=R_{n,k}$ in such cases.
    So we may assume that $0 \le k \le n$.
    Define $f\text{: } B_{k}(n) \to \APK{n+1}{k}$ by
     \[f(1^{k}0^{n-k}) = 01 \ldots k^{n+1-k} \text{ and }
        f(1^{k-x}0^{y}1^{x}0^{n-k-y}) = 01 \ldots k^{y}(k-x)^{n+1-k-y}
    \]
    where $1 \le x \le k$ and $1 \le y \le n-k$.
    By Lemma \ref{ascform}, all elements of $\APK{n+1}{k}$ have such a form.
    This map is invertible since we only need $x,y$ to reconstruct the original binary word, and $x,y$ can be recovered from the image.
    Therefore, $f$ is a bijection, yielding the desired result.
\end{proof}

With this bijection, all proofs of identities in Section \ref{identities} can be translated into the language of ascent sequences with some effort.
For example, a proof of Theorem \ref{rowsum} using ascent sequences is given in Proposition 9 of Baxter and Pudwell \cite{BP}.

%%%%%%%%
%%%%%%%%
%%%%%%%%
%%%%%%%%
%%%%%%%%
%%%%%%%%

\section{A Generalization}

We return briefly to our combinatorial interpretation of $R_{n,k}$ as enumerating $B_{k}(n)$ to offer a natural generalization.
We have shown that $R_{n,k}$ counts binary words with length $n$, $k$ 1's, and at most 1 ascent, but the choice to only count such binary words with at most 1 ascent is (more or less) arbitrary.
So, in this section, we consider a generalization where we allow for at most $j$ ascents. Let $B_{k}^{(j)}(n)$ be the set of all binary words of length $n$ with exactly $k$ 1's and at most $j$ ascents, and set $R_{n,k}^{(j)} = |B_{k}^{(j)}(n)|$ and $B^{(j)}(n) = \bigcup_{k} B_{k}^{(j)}(n)$. An equivalent generalization was considered in a recent paper by Gregory et al \cite{GKW}, where they discuss the numbers they refer to $\binom{n}{k}_j$, which are exactly what we call $R_{n,k}^{(j)}$. In that paper the authors give a different combinatorial interpretation for the generalized Rascal numbers---namely, they enumerate $k$-element subsets of $\{1,2,\ldots,n\}$ whose intersection with $\{1,2,\ldots, n-k\}$ contains at most $j$ elements---and they independently prove some of our following results, which we prove using our binary word interpretation. That said, Theorem \ref{genrowsum} is not given in the Gregory paper, and indeed it proves a conjecture from that paper. For completeness, Theorem \ref{gregbij} provides a bijection between the binary strings enumerated by $R_{n,k}^{(j)}$ and the restricted subsets enumerated by $\binom{n}{k}_j$.

With this generalization, it's natural to ask how (or if) our previous results apply to this generalization. To begin, our previous use of a quadratic formula no longer holds for proving unimodality, so a more computational proof would be needed to show this fact, and Gregory et al \cite{GKW} give such a proof; on the other hand, these numbers are symmetric, and the proof is identical to the $j=1$ case (see Theorem \ref{sym}). Indeed, we expect that many of the identities in Section \ref{identities} can be generalized to hold for $R_{n,k}^{(j)}$ as well, however, the proofs seem unlikely to differ greatly from the $j=1$ case. As such, we leave it to the reader to explore them further. That said, there are a couple of results which  do differ enough to warrant special mention, and there is also one result that is left open to the reader. To wit, the next result yields a recurrence that is nearly identical to Equation (\ref{fleron}), whose proof follows a similar idea to that in Theorem \ref{rec1}.

\begin{theorem}[Generalized Fleron's Recurrence]
    Set $R_{n,n}^{(j)} = R_{n,0}^{(j)} = 1$ for $n \ge 0$ and $R_{n,k}^{(j)} = 0$ when $n, k < 0$ or $k > n$.
    Also, $R_{n,k}^{(0)} = 1$ when $0 \le k \le n$.
    Then for $2 \le n$, $0 \le k \le n$, and $1 \le j$ we have \[
        R_{n,k}^{(j)} = R_{n-1,k}^{(j)} + R_{n-1,k-1}^{(j)} - R_{n-2,k-1}^{(j)} + R_{n-2,k-1}^{(j-1)}.
    \]
\end{theorem}
\begin{proof}
    Every $w \in B_{k}^{(j)}(n)$ either ends in a 0 or 1, and so define $T(w) = w_{1} \cdots w_{n-1}$ for any word $w = w_{1} \cdots w_{n}$.
    \begin{enumerate}
        \item[(1)] If $w$ ends in a 0 then $T(w) \in B_{k}^{(j)}(n-1)$.
        Since $T$ is an invertible map when acting on words ending in 0, there are $R_{n-1,k}^{(j)}$ such $w$.

        \item[(2)] If $w$ ends in a 1 then $T(w) \in B_{k-1}^{(j)}(n-1)$.
        Furthermore, $T(w)$ either ends in a 0 and has at most $j-1$ ascents or still ends in a 1 with at most $j$ ascents.
        So the only words in $B_{k-1}^{(j)}(n-1)$ which cannot be mapped to by $T$ acting on words ending in a 1 are those which end in a 0 and have exactly $j$ ascents.
        There are $R_{n-2,k-1}^{(j)}$ such words in $B_{k-1}^{(j)}(n-1)$ which end in 0, by the same argument in part (1), of which exactly $R_{n-2,k-1}^{(j-1)}$ have fewer than $j$ ascents.
        Thus, there are $R_{n-2,k-1}^{(j)} - R_{n-2,k-1}^{(j-1)}$ elements of $B_{k-1}^{(j)}(n-1)$ which end in a 0 and have $j$ ascents.
        Therefore, there are exactly \[
            R_{n-1,k-1}^{(j)} - (R_{n-2,k-1}^{(j)} - R_{n-2,k-1}^{(j-1)})
        \]
        such $w$.
    \end{enumerate}
    Summing the two cases gives our desired recurrence.
\end{proof}

Since this generalization satisfies a recurrence very similar to Fleron's recurrence, it seems natural that they should satisfy a recurrence similar to Equation (\ref{grec}), proven in Theorem \ref{genrec}.
We conjecture that \[
    R_{n,k}^{(j)} = \frac{R_{n-1,k}^{(j)}R_{n-1,k-1}^{(j)} + E(n,k,j)}{R_{n-2,k-1}^{(j)}}
\]
for some $E\text{: } \Z^{3} \to \{0,1,2\ldots\}$.
We have not been able to find a closed form for $E$ but, through numerical tests, we believe one to exist. While we may not have a closed form for $E$, we do have a closed form for $R_{n,k}^{(j)}$ in terms of binomial coefficients.

\begin{theorem}
    For $0 \leq n,k,j \in \Z$, \[
        R_{n,k}^{(j)} = \sum_{r=0}^{j} \binom{k}{i}\binom{n-k}{i}.
    \]
\end{theorem}
\begin{proof}
    Let $w \in B_{k}^{(j)}(n)$ such that $\operatorname{asc}(w) = r$.
    Then $w = 1^{x_{0}}0^{y_{1}}1^{x_{1}} \ldots 0^{y_{r}}1^{x_{r}}0^{y_{0}}$ with 
\[        \sum_{i=0}^{r} x_{i} = k \ \text{ and }
        \sum_{i=0}^{r} y_{i} = n-k,\]
    where $x_{0},y_{0} \ge 0$ and $x_{i},y_{i} > 0$ for $i >0$.
    Note that solutions to the above equations are in one-to-one correspondence with positive integral solutions to \[ \sum_{i=0}^{r} x_{i}' = k+1 \hspace{5mm} \text{and} \hspace{5mm} \sum_{i=0}^{r} y_{i}' = n-k+1, \] respectively.
    Thus, there are \[
        \binom{k+1-1}{i+1-1} = \binom{k}{i}
    \]
    solutions to left equation and \[
        \binom{n-k+1-1}{i+1-1} = \binom{n-k}{i}
    \]
    solutions to the right equation.
    Thus, there are $\binom{k}{i}\binom{n-k}{i}$ such $w$.
    Summing over all values of $r$ we get our desired result.
\end{proof}

Finally, we give a formula to compute the row sums of generalized Rascal triangles, and we note that this theorem also proves Conjecture 7.5 from Gregory et al \cite{GKW}, which we give as a corollary. 

\begin{theorem}[Generalized Row Sum]
    For $0 \leq n,j \in \Z$,   \[
        \sum_{k=0}^{n} R_{n,k}^{(j)} = \sum_{k=0}^{2j+1} \binom{n}{k}.
    \]
    \label{genrowsum}
\end{theorem}
\begin{proof}
    Note that the left hand side counts $B^{(j)}(n)$ by definition. 
    We can also count $B^{(j)}(n)$ by beginning with some $S \subseteq [n]$ where $|S| \le 2j+1$. Next, let $w = w_1 \ldots w_n$, place a divider before the $w_i$ for each $i \in S$, and label each divided section from left to right with $0, 1, \ldots, |S|$ in increasing order.
    For each section, if it is labeled with an even number then set all $w_i$ in the section equal to 1's, and if it is labeled with an odd number then set all $w_i$ in the section to 0's.
    Note that we can invert this process for a $w = w_{1} \ldots w_{n} \in B^{(j)}(n)$ by taking $\operatorname{Des}(w) \cup \operatorname{Asc}(w)$ and unioning the singleton set $\{n\}$ if $w_{1} = 0$.
    For a fixed $0 \le k \le 2j+1$, there are $\binom{n}{k}$ such $S$ with $|S|=k$, so by summing over all $k$ we get $|B^{(j)}(n)|$, thus achieving our desired result.
\end{proof}

\begin{corollary}[Conjecture 7.5 of \cite{GKW}]
    For all $j \in \Z$ with $j \ge 0$, \[
        \sum_{k=0}^{4j+3} R^{(j)}_{4j+3,k} = 2^{4j+2}.
    \]
\end{corollary}
\begin{proof}
    By Theorem \ref{genrowsum}, we have \[
        \sum_{k=0}^{4j+3} R^{(j)}_{4j+3,k} = \sum_{k=0}^{2j+1} \binom{4j+3}{k}
    \]
    which is the sum of the first half of the $(4j+3)$-rd row of pascal's triangle.
    Due to symmetry of binomial coefficients, this sum is half that of the whole row sum, which is $2^{4j+3}$.
\end{proof}

We also use Theorem \ref{genrowsum} to provide a slightly more direct proof of the following result.

\begin{corollary}[Theorem 7.3 of \cite{GKW}]
    For all $0 \leq n,j \in \Z$, \[
        \Delta^{2j+1}\left(\sum_{k=0}^{n} R_{n,k}^{(j)} \right) = 1
    \]
    where $\Delta$ is the forward difference operator with respect to $n$, i.e., for the sequence $(a_n)_{n \ge 0}$,
    $\Delta a_n = a_{n+1} - a_n$.
\end{corollary}
\begin{proof}
    By Theorem \ref{genrowsum}, we have \begin{align*}
        \Delta^{2j+1}\left(\sum_{k=0}^{n} R_{n,k}^{(j)} \right) &= \Delta^{2j+1}\left(\sum_{k=0}^{2j+1} \binom{n}{k} \right) \\
        &= \sum_{k=0}^{2j+1} \Delta^{2j+1}\binom{n}{k}
    \end{align*}
    where the last equality follows by linearity of $\Delta$.
    It is a well-know result that \[
        \Delta \binom{n}{k} = \binom{n}{k-1}.
    \]
    So, our sum becomes \begin{align*}
        \Delta^{2j+1}\left(\sum_{k=0}^{n} R_{n,k}^{(j)} \right) &= \sum_{k=0}^{2j+1} \binom{n}{k-(2j+1)} \\
        &= \binom{n}{0} = 1.
    \end{align*}
\end{proof}

\begin{theorem}
    For $n,j \in \Z$ with $n,j \ge 0$, \[
        \sum_{k=0}^{n} (-1)^{k}R_{n,k}^{(j)} = \begin{cases}
            0 & n \equiv 1 \pmod{2} \\
            (-1)^{j} \binom{n/2 - 1}{j} & n \equiv 0 \pmod{2}
        \end{cases}
    \] 
    with the convention $\binom{-1}{j} = (-1)^j$.
    \label{genaltrowsum}
\end{theorem}
\begin{proof}
    To prove this result, we will repeatedly use the involution principle to pair off words of opposite sign until the remaining words have a nice form which we can easily enumerate.
    First, define $\operatorname{wt}\text{: } B^{(j)}(n) \to B^{(j)}(n)$ by $\operatorname{wt}(w) = (-1)^{\#(\text{1's in $w$})}$.
    Next, define $I_0\text{: } B^{(j)}(n) \to B^{(j)}(n)$ as follows: if $w = 1^{x_0} w' 0^{y_0}$ where $w'$ does not begin in a 1 or end in a 0, then
    \begin{enumerate}
        \item[(1)] $I_0(w) = 1^{x_0 - 1} w' 0^{y_0 + 1}$ whenever $2 \nmid x_0$, 
        \item[(2)] $I_0(w) = 1^{x_0+1} w' 0^{y_0-1}$ whenever $2 \mid x_0$ and $y_0 > 0$, and 
        \item[(3)] $I_0(w) = w$ whenever $2 \mid x_0$ and $y_0 = 0$.
    \end{enumerate}
    Note that $I_0$ is a weight-preserving, sign-reversing involution and, for $w = 1^{x_0} w' 0^{y_0} \in \operatorname{fix}(I_0)$, $2 \mid x_0$ and $y_0 = 0$ so $w = 1^{2x_0'} w'$.
    Now, for $d > 0$, define $I_d\text{: } \operatorname{fix}(I_{d-1}) \to \operatorname{fix}(I_{d-1})$ as follows:
    if \[
        w = 1^{2x_0} \left( \prod_{i=1}^{\operatorname{asc}(w)} 0^{y_i}1^{x_i} \right) \in \operatorname{fix}(I_{d-1}),
    \]
    then \begin{enumerate}
        \item[(1)] $\displaystyle
            I_d(w) = 1^{2x_0} \left( \prod_{i=1}^{d-1} 0^{y_i}1^{x_i} \right) 0^{y_d + 1} 1^{x_d - 1} \left( \prod_{i=d+1}^{\operatorname{asc}(w)} 0^{y_i}1^{x_i} \right)$ whenever $2 \mid x_d$,
        
        \item[(2)] $\displaystyle
            I_d(w) = 1^{2x_0} \left( \prod_{i=1}^{d-1} 0^{y_i}1^{x_i} \right) 0^{y_d - 1} 1^{x_d + 1} \left( \prod_{i=d+1}^{\operatorname{asc}(w)} 0^{y_i}1^{x_i} \right)$ whenever $2 \nmid x_d$ and $y_d > 1$, and 
        \item[(3)] $I_d(w) = w$ whenever $2 \nmid x_d$ and $y_d = 0$ or $\operatorname{asc}(w) < d$.
    \end{enumerate}
    Note that $I_d$ is also a weight-preserving, sign-reversing involution for all $1 \le d \le j$.
    Furthermore, for $w \in \operatorname{fix}(I_j)$, we must have that $2 \nmid x_d$ and $y_d = 1$ for all $1 \le d \le j$, that is, \[
        w = 1^{2z_0} \left( \prod_{i=1}^{\operatorname{asc}(w)} 0^{1}1^{2z_i + 1} \right)
    \]
    with $z_i \ge 0$ and \[
        2z_0 + \sum_{i=1}^{\operatorname{asc}(w)} (2z_i+1) + \operatorname{asc}(w) = 2 \left( \sum_{i=0}^{\operatorname{asc}(w)} z_i + \operatorname{asc}(w) \right) = n.
    \]
    Note if $2 \nmid n$ then the above equation has no integer solutions and therefore there cannot exist a $w \in \operatorname{fix}(I_j)$ so our sum is 0 as desired.

    Now assume that $2 \mid n$.
    Note that picking the $z_i$'s in the above equation uniquely determines $w$, so we simply need to count solutions.
    Since solutions to the above are in 1-1 correspondence with solutions to \[
        \sum_{i=0}^{\operatorname{asc}(w)} z_i' = n/2 + 1
    \]
    with $z_i' > 0$, there are $\binom{n/2}{\operatorname{asc}(w)}$ solutions.
    Also, note that since $2\mid n$
    $\operatorname{wt}(w) = (-1)^{n-\operatorname{asc}(w)} = (-1)^{\operatorname{asc}(w)}$.
    Thus, we have \begin{align*}
        \sum_{k=0}^n (-1)^k R_{n,k}^{(j)} &= \sum_{\ell=0}^j (-1)^\ell \binom{n/2}{\ell} \\
        &= (-1)^j \binom{n/2-1}{j}
    \end{align*}
    as desired.

(Note that the $j=1$ case yields  Theorem \ref{altrowsum1}.)
\end{proof}

We conculde with a bijection between the binary strings enumerated by $R_{n,k}^{(j)}$ and the restricted subsets enumerated by $\binom{n}{k}_j$.

\begin{theorem} \label{binsubbijection}
    For $0 \leq n,k,j \in \Z$, we have \[
        R^{(j)}_{n,k} = \left| \binom{n}{k}_{j} \right|.
    \]
    \label{gregbij}
\end{theorem}
\begin{proof}
    Let $S \in \binom{n}{k}_{j}$ then set $S_x = S \cap \{n-k+1, \ldots, n\}$ and $S_y = S \cap \{1,2,\ldots,n-k\}$.
    Note that $S = S_x \cup S_y$ and $S_x \cap S_y = \emptyset$.
    Then, we can write $S_y$ as \[
        S_y = \left\{ y_1, \; y_1+y_2, \; y_1+y_2+y_3, \; \ldots, \; \sum_{i=1}^{m} y_i \right\}
    \]
    where $y_i > 0$ and $0 \le m \le j$.
    Define $h\text{: } \Z \to \Z$ by $a \mapsto a - (n-k)$ which is invertible and therefore a bijection.
    Then note that $h(S_x) \subset \{1,2,\ldots, k\}$ as for any $x \in S_x$, $n-k < x \le n$.
    Accordingly, we  write $h(S_x)$ as \[
        h(S_x) = [k] \setminus \left\{ x_1, \; x_1+x_2, \; \ldots, \; \sum_{i=1}^{m} x_i \right\}
    \]
    with $x_i > 0$.
    Now, define $f\text{: } \binom{n}{k}_j \to B^{(j)}_{k}(n)$ by \[
        f(S) = f(S_x \cup S_y) = 1^{k-X} \left( \prod_{i=1}^{m} 0^{y_i}1^{x_i} \right) 0^{n-k-Y}
    \]
    where $X = \max S_x$ and $Y = \max S_y$.
    Given a $w \in B^{(j)}_{k}(n)$, we can easily determine it's corresponding $x_i$ and $y_i$ values which gives us $S_y$ and $h(S_x)$.
    Since $h$ is invertible, we know $S_x$ as well, which gives us the $S \in \binom{n}{k}_j$ that is mapped to $w$ by $f$.
    Therefore, $f$ is invertible, completing the proof.
\end{proof}


\begin{thebibliography}{20}
\bibitem{ALT} A. Anggoro, E. Liu, \& A. Tulloch, The Rascal Triangle, \emph{The College Mathematics Journal}, \textbf{41.5} (2010), pp. 393-395.

\bibitem{Ash} J. M. Ashfaque, The Rascal Numbers, Available at \url{https://www.researchgate.net/publication/335686803_The_Rascal_Numbers}, July 11, 2023.

\bibitem{BP} A. Baxter \& L.  Pudwell, Ascent sequences avoiding pairs of patterns, \emph{The Electronic Journal of Combinatorics}, \textbf{22(1)} (2015), \#P1.58.

\bibitem{B-M} M. Bousquet-M\'elou, A. Claesson, M. Dukes, \& S. Kitaev, $(2+2)$-free posets, ascent sequences and pattern avoiding permutations, \emph{Journal of Combinatorial Theory, Series A}, \textbf{117(7)} (2010), pp. 884-909

\bibitem{DS} P. Duncan \& E. Steingr\'imsson, Pattern avoidance in ascent sequences, \emph{The Electronic Journal of Combinatorics} \textbf{18(1)} (2011), \#P2.26.

\bibitem{GKW} J. Gregory, B. Kronholm, \& J. White, Iterated rascal triangles, \emph{Aequationes Mathematicae}, (2023), \url{https://doi.org/10.1007/s00010-023-00987-6}

\bibitem{Fle} J. F. Fleron, Fresh perspective bring discoveries, \emph{Math Horizons}, \textbf{24} (2017), pp. 15.

\bibitem{Hot} P. K. Hotchkiss, Generalized rascal triangles, \emph{Journal of Integer Sequences}, \textbf{23} (2020), 20.7.4.

\bibitem{OEIS} N. J. A. Sloane et al, The Online Encyclopedia of Integer Sequences, 2023, Available at \url{https://www.oeis.org}.

\end{thebibliography}
\end{document}